\documentclass[a4paper,11pt,reqno]{amsart}

\usepackage[defaultlines=2,all]{nowidow}

\allowdisplaybreaks

\usepackage{ mathtools, amssymb, mathdots, bbm, mleftright, faktor, stmaryrd, bm }

\usepackage{ letltxmacro }    
\usepackage{ xparse }

\usepackage{ tikz, tikz-cd }
\usetikzlibrary{shapes.misc, calc, patterns}

\usepackage[shortlabels]{ enumitem }

\usepackage{ standalone, graphicx }
\usepackage[section]{ placeins } 
\usepackage[margin=0.25cm]{ caption }
\usepackage{ marginnote }

\usepackage{ amsthm, thmtools, hyperref, url }
\usepackage[nameinlink]{ cleveref }

\numberwithin{equation}{section}
\newtheorem{theorem}{Theorem}[section]
\newtheorem*{theorem*}{Theorem}
\newtheorem{corollary}[theorem]{Corollary}

\newtheorem{proposition}[theorem]{Proposition}
\newtheorem{conjecture}[theorem]{Conjecture}

\newtheorem{lemma}[theorem]{Lemma}

\theoremstyle{definition}
\newtheorem{definition}[theorem]{Definition}

\newtheorem{example}[theorem]{Example}

\newcounter{subequation}

\newcounter{thmlistcnt}

	{\setcounter{thmlistcnt}{0}%
	\begin{list}{\emph{(\roman{thmlistcnt})}}{%
		\usecounter{thmlistcnt}%
		\setlength{\topsep}{0pt}%
		\setlength{\leftmargin}{0pt}%
		\setlength{\itemsep}{0pt}%
		\setlength{\labelwidth}{17pt}
		\setlength{\itemindent}{30pt}}%
	}%
	{\end{list}}%
	
\newcounter{exlistcnt}
\newenvironment{exlist}%
	{\setcounter{exlistcnt}{0}%
	\begin{list}{(\alph{exlistcnt})}{%
		\usecounter{exlistcnt}%
		\setlength{\topsep}{0pt}%
		\setlength{\leftmargin}{0pt}%
		\setlength{\itemsep}{0pt}%
		\setlength{\labelwidth}{0pt}
		\setlength{\itemindent}{15pt}}%
	}%
	{\end{list}}%
  
\newcommand{\Z}{\mathbb{Z}}

\newcommand{\C}{\mathbb{C}}
\newcommand{\F}{\mathbb{F}}
\newcommand{\N}{\mathbb{N}}

\usepackage[vcentermath,enableskew]{youngtab} 

\DeclareMathOperator{\Sym}{Sym}

\renewcommand*{\det}{\qopname\relax o{det}}

\DeclareMathOperator{\GL}{GL}
\DeclareMathOperator{\SL}{SL}
\renewcommand{\sl}{\mathsf{sl}}

\renewcommand{\epsilon}{\varepsilon}
\renewcommand{\phi}{\varphi}

\newcommand{\mbf}[1]{\mathbf{#1}}

\newcommand{\qbinom}[2]{\genfrac{[}{]}{0pt}{}{#1}{#2}}

\newcommand{\E}{\!E}

\newcommand{\B}{\mathcal{B}}
\newcommand{\CC}[2]{\mathcal{C}_{#1}^{(#2)}}
\newcommand{\CCs}[2]{\mathcal{C}_{#1}^{\scalebox{0.75}{$\scriptscriptstyle (#2)$}}}

\DeclareMathOperator{\End}{End}
\newcommand{\rev}{\mathrm{rev}}

\definecolor{lightblue}{rgb}{0.75,0.75,1}
\definecolor{lightred}{rgb}{1,0.75,0.75}
\definecolor{lightgreen}{rgb}{0.5,1,0.5}
\definecolor{verylightgray}{rgb}{0.9,0.9,0.9}

\definecolor{orange}{rgb}{1,0.5,0.25}

\newcommand{\NB}{\mathcal{P}}
\newcommand{\CN}{\textup{c}}
\newcommand{\cp}[2]{\bigl( #1, #2 \bigr)}
\newcommand{\ddp}[2]{\rotatebox{90}{\scalebox{0.8}{$\bigl(#1, #2 \bigr)$}}}
\newcommand{\bigquad}{\hspace*{13pt}}

\newcommand{\I}[2]{I^{(#1)}(#2)}
\renewcommand{\P}[3]{F(#2,#3)}   
\newcommand{\PA}[2]{F\bigl(#2,#1-|#2|\bigr)} 

\newcommand{\vt}{t} 

\begin{document}

\title[A modular plethystic isomorphism]{A new modular plethystic $\SL_2(\F)$-isomorphism 
$\Sym^{N-1}\E \otimes \bigwedge^{N+1} \Sym^{d+1}\E \cong \Delta^{(2,1^{N-1})} \Sym^d\E$}


\author{Alvaro L. Martinez and Mark Wildon}
\email{martinez@math.columbia.edu}
\email{mark.wildon@bristol.ac.uk}



\newcommand{\dV}{d} 
\maketitle
\thispagestyle{empty}

\vspace*{-5pt} 
\begin{abstract}
Let $\F$ be a field and let $E$ be the natural representation of $\SL_2(\F)$. Given a vector space $V$, let $\Delta^{(2,1^{N-1})}V$ be the kernel of the multiplication map $\bigwedge^N \!V \otimes V \rightarrow \bigwedge^{N+1}\!V$. We construct an explicit $\SL_2(\F)$-isomorphism $\Sym^{N-1}\E \otimes \bigwedge^{N+1} \Sym^{d+1}\E \cong \Delta^{(2,1^{N-1})} \Sym^d\E$. This $\SL_2(\F)$-isomorphism is a modular lift of the $q$-binomial identity $q^{\frac{N(N-1)}{2}}[N]_q \qbinom{d+2}{N+1}_q = s_{(2,1^{N-1})}(1,q,\ldots, q^d)$, where $s_{(2,1^{N-1})}$ is the Schur function for the partition $(2,1^{N-1})$. This identity, which follows from our main theorem, implies the existence of an isomorphism when $\F$ is the field of complex numbers but it is notable, and not typical of the general case, that there is an explicit isomorphism defined in a uniform way for any field.
\end{abstract}


\section{Introduction}

Let $\F$ be an arbitrary field and 
let $E$ be the natural $2$-dimensional representation of the special linear group $\SL_2(\F)$.
Let $\Delta^\lambda$ denote the Schur functor canonically labelled by the partition $\lambda$.
Working over the field of complex numbers there is a rich theory of plethystic isomorphisms between
the representations $\Delta^\lambda \Sym^d \E$. These include Hermite reciprocity and the Wronskian isomorphism;
we refer the reader to \cite{PagetWildonSL2}
for a comprehensive account and references
to earlier results.
In \cite{McDowellWildon} it was shown
that both these classical
isomorphisms hold over an arbitrary field, \emph{provided that suitable dualities are introduced}.
The modular version of Hermite reciprocity is $\Sym_M \Sym^d \E \cong \Sym^d \Sym_M \E$,
where, given a $\SL_2(\F)$-representation $V$,
$\Sym^r \hskip-0.5pt V$ is the symmetric power defined as a quotient of $V^{\otimes r}$
and $\Sym_r \!V$ is its dual defined as the subspace of invariant tensors in $V^{\otimes r}$.
The modular Wronskian isomorphism is $\Sym_M \Sym^d \E \cong \bigwedge^M \Sym^{d+M-1} \E$. 
The purpose of this article is to add to the  collection of such modular plethystic isomorphisms
by proving the following theorem.
The version of the
Schur functor $\Delta^{(2,1^{N-1})}$ we require is defined in \S\ref{subsec:preliminaries} immediately below.

\begin{theorem}\label{thm:main}
Let $N \in \N$ and let $d \in \N_0$.
The map $\phi$ defined in Definition~\ref{defn:phi} is an
isomorphism of $\SL_2(\F)$-representations
\[ \Sym^{N-1}\E \otimes \bigwedge^{N+1} \Sym^{d+1}\E \cong \Delta^{(2,1^{N-1})} \Sym^d\E.\]
\end{theorem}

This theorem is notable as the first explicit example of a modular plethystic isomorphism
involving a Schur functor for a partition that is not one-row or one-column, and also for the
unexpected tensor factorisation it exhibits.

This isomorphism is a modular lift of the $q$-binomial identity
\begin{equation}
\label{eq:qCharactersSame}
q^{\frac{N(N-1)}{2}} [N]_q \qbinom{d+2}{N+1}_q = s_{(2,1^{N-1})}(1,q,\ldots, q^d)
\end{equation}
where $s_{(2,1^{N-1})}$ denotes the Schur function labelled by the partition $(2,1^{N-1})$.
We structure our proof so that we can obtain~\eqref{eq:qCharactersSame} as a fairly
routine corollary of Theorem~\ref{thm:main}: see Corollary~\ref{cor:qChar}, where we also give combinatorial
interpretations of each side.
As shown in \cite[Theorem 1.6]{McDowellWildon}
there exist representations of the form $\Delta^\lambda \Sym^d E$ 
that have equal $q$-characters in the sense of~\eqref{eq:qCharactersSame}, and so are isomorphic over~$\C$, 
but fail to be isomorphic over arbitrary fields $\F$, even after
considering all possible dualities. Indeed, the authors believe this is the generic case.
This adds to be interest and importance of Theorem~\ref{thm:main}.
We finish with Corollary~\ref{cor:phiGL}, which lifts the isomorphism in Theorem~\ref{thm:main}
to an isomorphism of representations of $\GL_2(\F)$, and Conjecture~\ref{conj:phiGeneral} on a
conjectured more general isomorphism.

\subsection{Preliminaries}\label{subsec:preliminaries}
Fix a basis $X$, $Y$ of the $\F$-vector space $E$. For each $r \in \N_0$,
the symmetric power $\Sym^c \!E$ has as a basis the monomials $X^{c-i}Y^i$ for
$0\le i \le c$. 

\subsubsection*{Schur functor}
It will be convenient to define the Schur functor $\Delta^{(2,1^{N-1})}$ 
on a vector space $V$ by
\begin{equation}\label{eq:muN} 
\Delta^{(2,1^{N-1})} \,V = \ker \mu_N : \bigwedge^N V \otimes V \rightarrow \bigwedge^{N+1}V \end{equation}
where $\mu_N$ is the multiplication map  $v_1 \wedge \cdots \wedge v_N \otimes w \mapsto
v_1 \wedge \cdots \wedge v_N \wedge w$.
Thus $\Delta^{(2,1^{N-1})}\, V$ is a subspace of $\bigwedge^N \hskip-0.5pt V \otimes V$.
Since $\mu_N$ is a homomorphism of representations of $\GL(V)$, for any fixed group $G$,
$\Delta^{(2,1^{N-1})}$ is a functor on the category of $\F$-representations of $G$.

\subsubsection*{Multi-indices}
For $c$, $r \in \N_0$, let $\I{c}{r}$ denote the set $\{0,1,\ldots, c\}^r$. We say
that the elements of $\I{c}{r}$ are \emph{multi-indices}.
We define the \emph{sum} of a multi-index $\mbf{i}$ by $|\mbf{i}| = \sum_{\alpha=1}^r i_\alpha$.
Given $\mbf{i} \in \I{c}{r}$ we define
\[ F^{(c)}_\wedge(\mbf{i}) = X^{c-i_1}Y^{i_1} \wedge \cdots \wedge X^{c-i_r}Y^{i_r}. \]
Thus $\bigwedge^r \Sym^c \E$ has as a basis all $F^{(c)}_\wedge(\mbf{i})$ for 
strictly increasing $\mbf{i} \in \I{c}{r}$.
We say that a pair $(\mbf{i}, j) \in \I{c}{r} \times \{0,1,\ldots, c\}$
is \emph{semistandard} if $\mbf{i}$ is strictly increasing and $i_1 \le j$.
Observe that $(\mbf{i}, j)$ is semistandard if and only if the $(2,1^{N-1})$-tableau $t_{(\mbf{i},j)}$ shown in the margin
having
entry $i_\alpha$ in box $(\alpha,1)$ and entry $j$ in box $(1,2)$ is semistandard in the usual sense.
\marginpar{\raisebox{-12pt}{$\ \ \raisebox{48pt}{$t_{(\mbf{i},j)} =$} \, \, \begin{tikzpicture}[x=0.65cm, y=-0.65cm]\draw(0,0)--(2,0)--(2,1)--(0,1)--(0,0); 
\draw(0,0)--(0,5)--(1,5)--(1,0); \draw (0,2)--(1,2); \draw (0,4)--(1,4);
\node at (0.5,0.5) {$i_1$}; 
\node at (1.5,0.5) {$j$};
\node at (0.5,1.5) {$i_2$};
\node at (0.5,2.75) {$\vdots$};
\node at (0.5,4.5) {$i_N$}; \end{tikzpicture}$}}

\begin{definition}[Content and Neighbour]\label{defn:neighbour}
Let $d \in \N_0$.
Let $\mbf{i} \in \I{d}{N}$ and $j \in \{0,1,\ldots, d\}$ and let $(\mbf{i},j)$
be semistandard. We define the \emph{content}
of $(\mbf{i}, j)$ to be the multiset $\{i_1,\ldots, i_N \} \cup \{ j \}$.
We define the \emph{neighbour} of $(\mbf{i}, j)$ by
\begin{equation}
\label{eq:neighbour}
\NB(\mbf{i}, j) = \bigl( (i_1, \ldots, i_{\alpha-1}, j, i_{\alpha+1}, \ldots, i_N), i_\alpha \bigr)
\end{equation}
where $\alpha \in \{1,\ldots, N\}$ is maximal such that $i_\alpha \le j$.
\end{definition}

Observe that the neighbour map is well-defined because $i_1 \le j$ and that it
preserves content. Moreover,
$\NB(\mbf{i}, j) = (\mbf{i}, j)$ if and only if $j$ is in the set $\{i_1,\ldots, i_N\}$;
in this case $j$ is the unique repeated element of the multiset.
For example, repeated applications of the neighbour map give
\begin{equation}\label{eq:chain} 
\bigl( (0,2,3), 5 \bigr) \stackrel{\NB}{\longmapsto} \bigl((0,2,5), 3\bigr)
\stackrel{\NB}{\longmapsto}  \bigl((0,3,5), 2\bigr)\stackrel{\NB}{\longmapsto}  \bigl((2,3,5), 0\bigr) 
\end{equation}
in which all pairs have content $\{0,2,3,5\}$
and the final pair is the only one that is not semistandard
and so does not have a defined neighbour.

\begin{definition}\label{defn:F}
Let $d \in \N_0$.
Let $\mbf{i} \in \I{d}{N}$ and $j \in \{0,1,\ldots, d\}$. We define
$F(\mbf{i},j) \in \bigwedge^N \hskip-0.5pt\Sym^d\E \otimes \Sym^d \E$ by
\begin{equation}\label{eq:F} 
F(\mbf{i},j) = F^{(d)}_\wedge (\mbf{i}) \otimes  X^{d-j}Y^{j}. \end{equation}
If $(\mbf{i}, j)$ is semistandard we define
\begin{equation}\label{eq:FDelta} F_\Delta(\mbf{i},j) =
\begin{cases} F(\mbf{i},j) & \text{if $j \in \{i_1,\ldots, i_N\}$} \\
	F(\mbf{i},j) + F\bigl( \NB(\mbf{i}, j) \bigr) 
& \text{otherwise.}\end{cases} 
\end{equation}
\end{definition}

Whenever we use the notation $F(\mbf{i},j)$ in~\eqref{eq:F}, the value of $d$ will be clear from context.
To give an example  we take $N=3$ and $d=5$.
Then, omitting some parentheses for readability, we have
$F\bigl((0,2,5),3\bigr) = 
 X^5 \wedge X^3Y^2 \wedge Y^5 \otimes X^2Y^3$ and
\begin{align*} 
F_\Delta\bigl((0,2,5),3\bigr) &= 
F\bigl((0,2,5), 3 \bigr) + F\bigl((0,3,5),2\bigr) \\
&= X^5 \wedge X^3Y^2 \wedge Y^5 \otimes X^2Y^3
+ X^5 \wedge X^2Y^3 \wedge Y^5 \otimes X^3Y^2. \end{align*}

\subsubsection*{Semistandard basis}
It is clear that $\bigwedge^N \Sym^d \E \otimes \Sym^d \E$ has as a canonical basis all $F(\mbf{i},j)$ for
strictly increasing $\mbf{i} \in \I{d}{N}$
and $j \in \{0,1,\ldots, d\}$. 
In this subsection we show
that the $F_\Delta(\mbf{i}, j)$ for semistandard $(\mbf{i}, j)$ form a basis for its submodule $\Delta^{(2,1^{N-1})} \Sym^d E$,
which we defined to be the kernel of $\mu_N :  \bigwedge^N V \otimes V \rightarrow \bigwedge^{N+1}V$
in~\eqref{eq:muN}.

We first show that these elements are in the kernel.
Let $\mbf{i} \in \I{d}{N}$ be strictly increasing
and let $j \in \{0,1,\ldots, d\}$. Let $\alpha$ be maximal
such that $i_\alpha \le j$. Then
\begin{align*}
 \mu_N  F \bigl( \NB(\mbf{i}, j) \bigr) &= 
\mu_N F\bigl( (i_1, \ldots, i_{\alpha-1}, j, i_{\alpha+1}, \ldots, i_N), i_\alpha \bigr) \\
&= X^{d-i_1}Y^{i_1} \wedge \cdots \wedge X^{d-j}Y^j \wedge \cdots
\wedge X^{d-i_N}Y^{i_N} \wedge X^{d-i_\alpha}Y^{i_\alpha}. 
\end{align*}
Up to  a swap of the factor  $X^{d-j}Y^j$ in position $\alpha$
and the final factor $X^{d-i_\alpha}Y^{i_\alpha}$, the right-hand side agrees with $\mu_N \bigl( F(\mbf{i}, j) \bigr)$.
Hence, by the definition of $F_\Delta$ in~\eqref{eq:FDelta}, we have
$\mu_N F_\Delta(\mbf{i},j) = 0$.
In order to prove the basis theorem, we shall use the following lemmas.

\begin{lemma}\label{lemma:chain}
Let $\{a_1,\ldots, a_N, b \}$ be a subset of $\{0,1,\ldots, d\}$
with $a_1 < \ldots < a_N < b$.
The $N$ distinct semistandard
pairs with content $\{a_1,\ldots, a_N,b\}$ are
$\mathcal{P}^\alpha\bigl( (a_1,\ldots,a_N), b \bigr)$ for $\alpha\in\{0,1,\ldots,N-1\}$; the second elements of these pairs form
the decreasing chain $b > a_N > \ldots > a_2$.
\end{lemma}

\begin{proof}
We leave this to the reader as a routine generalisation of~\eqref{eq:chain}.
\end{proof}

\begin{lemma}\label{lemma:DeltaLinearlyIndependent} 
The $F_\Delta(\mbf{i}, j)$ for semistandard $(\mbf{i}, j) \in \I{d}{N} \times
\{0,1,\ldots, d\}$ are linearly independent.
\end{lemma}

\begin{proof}
It is clear from the canonical basis of $\bigwedge^N \Sym^d E \otimes \Sym^d E$
that 
if there is a non-trivial linear relation
between the $F_\Delta(\mbf{i}, j)$
then it may be assumed to involve only $(\mbf{i}, j)$ of the same
content (in the sense of Definition~\ref{defn:neighbour}). If $j \in \{i_1,\ldots, i_N\}$ then there
is a unique basis element $F_\Delta(\mbf{i}, j)$ of content $\{i_1,\ldots, i_N\}
\cup \{j\}$. In the remaining case, the content multiset is a set, $A$ say,
and Lemma~\ref{lemma:chain} applies. Since each $F_\Delta(\mbf{i}, j)$ of content $A$ is 
a sum of $F(\mbf{i}, j)$ and its neighbour $F\bigl( \NB(\mbf{i}, j) \bigr)$, 
they are linearly independent.
\end{proof}


\begin{lemma}\label{lemma:combinatorial}
    There are $N\binom{d+2}{N+1}$ semistandard Young tableaux  of shape $(2,1^{N-1})$ 
    and entries in $\{0,1,\ldots,d\}$. 
\end{lemma}

\begin{proof}
Let $S_\alpha$ be the set of all Young tableaux $t_{(\mbf{i},j)}$ 
such that $i_\alpha \le j < i_{\alpha + 1}$ if $\alpha < N$, and such that $i_N \le j$ if $\alpha=N$. It is clear that
the set of semistandard Young tableaux of shape $(2,1^{N-1})$ is partitioned into the $N$ disjoint subsets $S_1,\ldots,S_N$.
We claim that each $S_\alpha$ has the same cardinality $\binom{d+2}{N+1}$.
To see this, we define a bijection from $S_\alpha$ to the set of strictly increasing multi-indices in 
\smash{$\I{d+1}{N+1}$} by
\[ (\mbf{i},j) \mapsto (i_1, \ldots, i_\alpha, j+1, i_{\alpha+1}+1, i_{\alpha+2}+1,\ldots, i_N+1).\]
The inverse of this map is easily shown to be
\[(k_1, \ldots, k_{N+1}) \mapsto \bigl((k_1, \ldots, k_\alpha, k_{\alpha+2} - 1, \ldots, k_{N+1} - 1 ), 
k_{\alpha +1} -1\bigr). \qedhere \]
\end{proof}

\begin{proposition}\label{prop:dims}
    The vector space $\Delta^{(2,1^{N-1})}\Sym^d \E$ has dimension $N\binom{d+2}{N+1}$ and 
    has as a basis 
\[ \bigl\{ F_\Delta(\mbf{i}, j) : \mbf{i} \in \I{d}{N}, j \in \{0,1,\ldots, d\},
\text{$(\mbf{i}, j)$ semistandard} \bigr\}. \]
\end{proposition}

\begin{proof}
By Lemma~\ref{lemma:DeltaLinearlyIndependent},
the claimed basis is linearly independent.
We use a dimension counting argument to show that it spans $\Delta^{(2,1^{N-1})}\Sym^d \E$.
By Lemma \ref{lemma:combinatorial}, it suffices to show that $\dim{\Delta^{(2,1^{N-1})}\Sym^d \E}=N\binom{d+2}{N+1}$.
This follows from the rank-nullity formula applied to~\eqref{eq:muN}:
\begin{align*}
\dim \ker \mu_N &=\binom{d+1}{N} (d+1) - \binom{d+1}{N+1}\\
&= \binom{d+1}{N} (d+2) -  \left(\binom{d+1}{N} + \binom{d+1}{N+1} \right)\\
 &= \binom{d+2}{N+1} (N+1) - \binom{d+2}{N+1}\\
 & = N\binom{d+2}{N+1}. \qedhere 
\end{align*}
\end{proof}


\subsection{Definition of $\phi$}
Given $0 \le j < k$, set $[j,k) = \{j,j+1,\ldots, k-1\}$.
Given a strictly increasing multi-index $\mbf{k} \in \I{d+1}{N+1}$, we define
\begin{equation}\label{eq:box} \B(\mbf{k}) = [k_1,k_2) \times [k_2, k_3) \times \cdots \times [k_N, k_{N+1})
\subseteq \I{d}{N}. \end{equation}
For example if $d=5$ and $N=3$, we have \[ \B\bigl( (0,2,3,6) \bigr) = [0,2) \times [2,3) \times [3,6) = 
\{0,1\} \times \{2\} \times \{3,4,5\} \subseteq \I{5}{3}. \]

\begin{definition}\label{defn:phi}
Fix $d \in \N_0$ and $N \in \N$. We define 
\[ \phi : \Sym^{N-1}\E \otimes \bigwedge^{N+1} \Sym^{d+1}\E \rightarrow 
\bigwedge^N \Sym^d E \otimes \Sym^d E \] 
by
\begin{equation}\label{eq:phiAlt} \phi \bigl( X^{N-1-s}Y^s \otimes F^{(d+1)}_\wedge(\mbf{k})
\bigr) =  \sum_{\mbf{i} \in \B(\mbf{k})} F\bigl(\mbf{i}, s + |\mbf{k}|-N-|\mbf{i}|\bigr) \end{equation}
where $s \in \{0,1,\ldots, N-1\}$ and $\mbf{k} \in \I{d+1}{N+1}$ is strictly increasing.
\end{definition}

We remind the reader 
that, by~\eqref{eq:F} in Definition~\ref{defn:F}, the summand on the right hand side in~\eqref{eq:phiAlt}
is \smash{$F^{(d)}_\wedge(\mbf{i}) \otimes X^{d-(s+|\mbf{k}|-N-|\mbf{i}|)}
Y^{s+|\mbf{k}|-N-|\mbf{i}|}$}, or written out in full,
\[ X^{d-i_1}Y^{i_1} \wedge \cdots \wedge X^{d-i_N}Y^{i_N}
\otimes X^{d-(s+|\mbf{k}|-N-|\mbf{i}|)}
Y^{s+|\mbf{k}|-N-|\mbf{i}|}.\]
If $\mbf{i} \in \B(\mbf{k})$ then 
\begin{align}
|\mbf{i}| &\le (k_2-1) + \cdots + (k_{N+1} - 1) = |\mbf{k}| - k_1 - N \nonumber \\
|\mbf{i}| &\ge k_1 + \cdots + k_N = |\mbf{k}| - k_{N+1} \label{eq:kBounds} 
\end{align}
and so $s+|\mbf{k}| - N - |\mbf{i}| \ge s + k_1 \ge 0$ and 
$s+|\mbf{k}| - N - |\mbf{i}| \le s + k_{N+1} -N \le  k_{N+1}  - 1 \le d$.
Therefore $\phi$ is well-defined.
As motivation and an aide-memoire, we note that a canonical basis
element of $Y$-degree $s + |\mbf{k}|$ maps under $\phi$ to
a sum of canonical basis elements each of $Y$-degree $s + |\mbf{k}| - N$. 
It is not obvious that
$\phi$ has image in the subrepresentation $\Delta^{(2,1^{N-1})}\Sym^d \E$
of $\bigwedge^N \hskip-0.5pt V \otimes V$.
We give a short proof of this fact in
Lemma~\ref{lemma:imageInDelta}. 

\begin{samepage}
\begin{example}{\ }\label{ex:ex}
\begin{exlist}
\item By~\eqref{eq:phiAlt}, the canonical basis element
\[ X^{N-1} \otimes F_\wedge^{(d+1)}(0,1,\ldots, N) = X^{N-1} \otimes X^{d+1} \wedge X^d Y \wedge \cdots \wedge X^{d-N+1}Y^N \]
in $\Sym^{N-1}\E \otimes \bigwedge^{N+1} \Sym^{d+1}\E$
of minimal $Y$-degree $0 + 1+\cdots + N$ maps under $\phi$ to the canonical basis element
\[
 F\bigl( (0,1,\ldots, N\!-\!1), 0\bigr) =X^d \wedge X^{d-1}Y \wedge \cdots \wedge X^{d-N+1}Y^{N-1} \otimes X^d \]
in  $\Delta^{(2,1^{N-1})}\Sym^d E$ of minimal $Y$-degree $0 + 1+\cdots  + (N-1)$.
Working over $\C$, these 
vectors are highest weight for the action of the Lie algebra generator $e$ (which may be thought of
as $X \frac{\mathrm{d}}{\mathrm{d}Y}$)
in~\eqref{eq:ef}.
\item More generally the image of $X^{N-1-s}Y^s \otimes F_\wedge^{(d+1)}(i,i+1,\ldots,i+N)$ is
$F\bigl((i,i+1,\ldots,i+N-1),s+i\bigr)$. Note that since $s\in \{0,\ldots,N-1\}$, it follows
from~\eqref{eq:muN} that
this image is in $\Delta^{(2,1^{N-1})}\Sym^d\E$. 
\item The image of a canonical basis element of $\Sym^{N-1}\E \otimes \bigwedge^{N+1}\Sym^{d+1}\E$ 
typically has many summands.
For instance
take $N = 3$ and $d = 5$ and $(s, \mbf{k}) = \bigl(1 , (0,2,3,6) \bigr)$. Then
\begin{align*} \quad&\phi(X Y \otimes X^6 \wedge X^4Y^2 \wedge X^3Y^3 \wedge Y^6)
\\ &= \sum_{\mbf{i} \in [0,2) \times [2,3) \times [3,6) } F(\mbf{i}, 1 + |(0,2,3,6)| - 3 - |\mbf{i}|)  \\
&= F\bigl((0,2,3), 4\bigr) + F\bigl((1,2,3),3\bigr) + F\bigl((0,2,4),3\bigr) + F\bigl((1,2,4),2\bigr) \\
& \qquad + F\bigl((0,2,5),2\bigr) + F\bigl((1,2,5),1\bigr).\end{align*}
We invite the reader to check that the right-hand side is in $\Delta^{(2,1,1)}\Sym^5\E$.
\item Taking $N=1$ we may identify $\Sym^0 \E$ with $\mathbb{F}$ and
$\Delta^{(2)}\Sym^d \E$ with the symmetric tensors inside $\Sym^d \E \,\otimes\, \Sym^d \E$. 
The map $\phi : \bigwedge^2 \Sym^{d+1} E \rightarrow \Delta^{(2)} \Sym^d E$
is then defined by
\[ \phi (X^{d+1-k}Y^k \wedge X^{d+1-\ell}Y^\ell) = \sum_{k \le i < \ell} X^{d-i}Y^i \otimes X^{d-(k+\ell-1-i)}
Y^{k+\ell-1-i}.\]
It is clear from the powers of $Y$ in the tensor factors on the right-hand side that 
the right-hand side is
a symmetric tensor lying 
in $\Delta^{(2)}\Sym^d \E$. This is an example of the Wronskian isomorphism mentioned
at the start of the introduction.
\end{exlist}

\end{example}
\end{samepage}

\section{The map $\phi$ is an $\SL_2(\F)$-isomorphism}	

\subsection{The image of $\phi$} As defined $\phi$ has codomain $\bigwedge^N \Sym^d \E \otimes \Sym^d \E$.	

\begin{lemma}\label{lemma:imageInDelta}
The image of $\phi$ is contained in $\Delta^{(2,1^{N-1})}\Sym^d \E$.
\end{lemma}

\begin{proof}
Let $s \in \{0,1,\ldots, N-1\}$ and let $\mbf{k} \in \I{d+1}{N+1}$ be strictly increasing.
Let $\Omega$ be the subset of 
\[ \B(\mbf{k}) \times [0,d+1) = [k_1,k_2) \times \cdots \times [k_N, k_{N+1}) \times \{0,1,\ldots, d\}\]
of all tuples $(i_1,\ldots, i_N, j)$ such that $i_1 + \cdots + i_N + j = s + |\mbf{k}| - N$.
Writing elements of $\Omega$ as $(\mbf{i}, j)$, we have, using the notation of~\eqref{eq:F},
\begin{equation}\label{eq:muNimage} 
\mu_N \phi \bigl( X^{N-1-s}Y^s \otimes F^{(d+1)}_\wedge (\mbf{k}) \bigr) = \sum_{(\mbf{i}, j) \in \Omega} 
\mu_N F(\mbf{i}, j). \end{equation} 
By the definition of $\Delta^{(2,1^{N-1})}\Sym^d E$ from~\eqref{eq:muN}, it suffices
to show that the right-hand side vanishes. Our proof uses an involution on $\Omega$ 
related to the partition used to prove Lemma~\ref{lemma:combinatorial}. (See
after the proof for the connection with the neighbour map.)
First observe that if $(\mbf{i}, j) \in \Omega$
then 
$j = s + |\mbf{k}| - N - |\mbf{i}|$ 
and by~\eqref{eq:kBounds} we have
\[ k_1 \le j < k_{N+1}. \]
Thus, given $(\mbf{i}, j) \in \Omega$, there exists a unique $1\le \alpha \le N$ 
such that $k_\alpha \le j < k_{\alpha+1}$. 
We send $(\mbf{i},j)$ to $(i_1,\ldots, j,\ldots i_N, i_\alpha\bigr)$
where $j$ appears in position~$\alpha$, so $j$ and $i_\alpha$ are swapped.
It is clear this defines an involution in which 
$(\mbf{i}, j)$ is a fixed point if and only if $i_\alpha = j$.
Since $(i_1,\ldots, j, \ldots, i_N, i_\alpha)$ and $(i_1,\ldots, i_\alpha, \ldots, i_N, j)$ 
are either equal or 
differ by a transposition, their contributions to the sum in~\eqref{eq:muNimage} cancel.
\end{proof}

To illustrate the neighbour map in Definition~\ref{defn:neighbour}
we remark that the image of $(\mbf{i}, j)$ under the involution in the proof
of Lemma~\ref{lemma:imageInDelta}
is $\NB(\mbf{i}, j)$ if 
$k_\alpha \le i_\alpha \le j < k_{\alpha+1}$ and
$\NB^{-1}(\mbf{i}, j)$ if $k_\alpha \le j < i_\alpha < k_{\alpha+1}$.

\subsection{$\phi$ is an $\SL_2(\F)$-homomorphism}
For $\beta \in \N$ we denote by $\mbf{u}^{(\beta)}$ the unit vector
$(0,\ldots,1,\ldots,0)$ 
where the non-zero entry is in position $\beta$; the length is always $N$ or $N+1$ and will
always be clear from context.

\subsubsection*{Reduction}
We recall the technical trick in~\cite[\S 4.2]{McDowellWildon}
used to pass from $\SL_2(\F)$ to $\SL_2(\C)$. First notice that $\phi$ is a map of vector spaces, but it is defined over the integers. Let $\gamma \in \F$ be an arbitrary element and $U_\gamma = \begin{pmatrix}1 & \gamma \\ 0 & 1\end{pmatrix}$. The elements $U_\gamma$ and their transposes generate $\SL_2(\F)$.  Checking that $\varphi$ intertwines the action of $U_\gamma$ (or its transpose) amounts to an equality of polynomials in~$\gamma$ with coefficients in the image of $\Z$ in $\F$. Clearly, it suffices to check that this equality holds over the polynomial ring $\Z[\gamma]$. For this, in turn, it suffices to prove the equality for \textit{any} transcendental 
element~$\gamma$ in any field containing $\Z$ as a subring. Proving the result for $\SL_2(\C)$ certainly implies the latter condition. A basic fact from Lie theory (see for instance~\cite[Ch.~8]{FH}) then reduces the question to proving that $\phi$ commutes with the Lie algebra generators $e$ and $f$ of $\sl_2(\C)$, defined
on the $X$, $Y$ basis of $E$ by the matrices
\begin{equation} e = \left( \begin{matrix} 0 & 1 \\ 0 & 0 \end{matrix}\right), \quad
   f = \left( \begin{matrix} 0 & 0 \\ 1 & 0 \end{matrix}\right)\! . \label{eq:ef} \end{equation}
Their action on $\Sym^d \E$ is given by $e \cdot g = X \frac{\mathrm{d}g}{\mathrm{d}Y}$
and $f \cdot g = Y \frac{\mathrm{d}g}{\mathrm{d}X}$. Their action on $\bigwedge^R \Sym^{c}\E$
is then given by the usual multilinear rule for Lie
algebra actions, coming ultimately from 
\begin{equation}\label{eq:tensorAction}
x \cdot (u \otimes v) = (x \cdot u) \otimes v + u\otimes (x \cdot v).
\end{equation}
We state it below using the unit vectors $\mbf{u}^{(\gamma)}$ just defined:
\begin{align}
e \cdot F_\wedge^{(c)}(\mbf{k}) &= \sum_{\alpha = 1}^{R} \label{eq:eOnFWedge}
k_\alpha F_\wedge^{(c)}(\mbf{k} - \mbf{u}^{(\alpha)}) \\ 
f \cdot F_\wedge^{(c)}(\mbf{k}) &= \sum_{\alpha = 1}^{R} 
(c-k_\alpha) F_\wedge^{(c)}(\mbf{k} +  \mbf{u}^{(\alpha)})\nonumber  
\end{align}
for $\mbf{k} \in \I{c}{R}$. 
Here we use the convention that if \hbox{$\mbf{k} \pm \mbf{u}^{(\alpha)} \!\not\in\! \I{d+1}{N+1}$}
because either $k_\alpha < 0$ or $k_\alpha > d+1$ then \smash{$F_\wedge^{(d+1)}(\mbf{k}\pm \mbf{u}^{(\alpha)}) = 0$}. 
An application of~\eqref{eq:tensorAction} now
gives the Lie algebra action on $\Delta^{(2,1^{N-1})}\Sym^d E \subseteq
\bigwedge^N \Sym^d E \otimes \Sym^d E$, which we use in the proof of Lemma~\ref{lemma:e} below.

\subsubsection*{Technical lemma}
The following lemma isolates the key step 
in the calculation that~$\phi$ commutes with the Lie algebra
action of $e$. In it we write $\B(\mbf{k}) - \mbf{u}^{(\beta)}$ 
for $\bigl\{ \mbf{j} - \mbf{u}^{(\beta)} : \mbf{j}  \in \B(\mbf{k}) \bigr\}$.
The notation $F(\mbf{i}, j)$ was introduced in Definition~\ref{defn:F}.
The first
paragraph of the proof below that checks that, in every summand, the second
component of the pair is in
$\{0,1,\ldots, d\}$, and so the expression is well-defined.
 This
technical check could be skipped by instead regarding the $F(\mbf{i}, j)$ 
as formal symbols; the proof then goes through. 

\begin{lemma}\label{lemma:technical}
Let $s \in \{0,1,\ldots, N-1\}$.
Let $\mbf{k} \in \I{d+1}{N+1}$ be strictly increasing. 
Then for any $\vt$ such that $|\mbf{k}| - N - 1 \le t \le |\mbf{k}| - 2$ we have
\[ \begin{split} &
\sum_{\alpha=1\rule{0pt}{6.75pt}}^{N+1} \sum_{\mbf{i} \in \B(\mbf{k}-\mbf{u}^{(\alpha)})} \!\!\!
k_\alpha \PA{\vt}{\mbf{i}} \\
&\qquad = \sum_{\beta=1\rule{0pt}{6.75pt}}^N \sum_{\mbf{j} \in \B(\mbf{k}) - \mbf{u}^{(\beta)}} \!\!\!\! 
(j_\beta+1)\PA{\vt}{\mbf{j}}
+ \sum_{\mbf{j} \in \B(\mbf{k})} \! \bigl( |\mbf{k}| -N - |\mbf{j}| \bigr) 
\PA{\vt}{\mbf{j}}. \end{split}
\]
\end{lemma}

\begin{proof}
If $\mbf{i} \in \B(\mbf{k} - \mbf{u}^{(\alpha)})$
then, by replacing $\mbf{k}$ with $\mbf{k} - \mbf{u}^{(\alpha)}$ 
in~\eqref{eq:kBounds} we have 
$|\mbf{k}| - k_{N+1} -1 \le |\mbf{i}| \le |\mbf{k}| - k_1 - N-1$
and so
\[ |\mbf{k}| - (d+1)-1 \le |\mbf{i}| \le |\mbf{k}| - N-1.\]
Hence $t - |\mbf{k}| + N+1 \le t- |\mbf{i}| \le t - |\mbf{k}| + d+2$
and the hypothesis on $t$ implies that $0 \le t - |\mbf{i}| \le d$.
Thus each $F(\mbf{i}, t-|\mbf{i}|)$ is well-defined. This also
shows each $F(\mbf{j}, t-|\mbf{j}|)$ in the first summand on the right-hand
side is well-defined; in the second summand we have instead
$t - |\mbf{k}| + N \le t- |\mbf{j}| \le t - |\mbf{k}| + d+1$
and now if $t = |\mbf{k}| - N - 1$ we may have $-1 = t - |\mbf{j}|$,
but in this case the coefficient is $|\mbf{k}| - N - |\mbf{j}| 
= |\mbf{k}| - N - (1 + t) = 0$, so the ill-defined summand can be ignored.
We are now ready to begin the main part of the proof.

Given $x \in \{0,1,\ldots, d\}$ and $1 \le \alpha \le N$, we set
\[ \CC{\alpha}{x}(\mbf{k}) = [k_1,k_2) \times \cdots
\times [k_{\alpha-1},k_\alpha)  \times \{x\} \times [k_{\alpha+1},k_{\alpha+2}) \times \cdots \times [k_N, k_{N+1}) \]
where $\{x\}$ in position $\alpha$ replaces the interval $[k_\alpha,k_{\alpha+1})$ in position
$\alpha$ of the product defining $\B(\mbf{k})$ in~\eqref{eq:box}.
Observe that
\begin{align}
\B(\mbf{k}\!-\!\mbf{u}^{(1)}) &= [k_1-1,k_2) \!\times\! [k_2,k_3) \!\times\! \cdots \!\times\! [k_N, k_{N+1}) 
= \B(\mbf{k}) \cup \CC{1}{k_1-1}, \nonumber \\
\!\B(\mbf{k}\!-\!\mbf{u}^{(N+1)}) &= [k_1,k_2) \!\times\! \cdots \!\times\! [k_{N-1},k_N)
\!\times\! [k_N,k_{N+1}\!-\!1) = \B(\mbf{k}) \,\backslash\, \CC{N}{k_{N+1}-1} \nonumber
\intertext{and, if $2 \le \alpha \le N$, then}
\B(\mbf{k}\!-\!\mbf{u}^{(\alpha)}) &= [k_1,k_2) \!\times\! \cdots \!\times\! [k_{\alpha-1},k_\alpha\!-\!1)
\!\times\! [k_\alpha\!-\!1, k_\alpha) \!\times\! \cdots \times [k_N,k_{N+1}) \nonumber \\
&= \B(\mbf{k}) \cup \CC{\alpha}{k_\alpha-1}(\mbf{k}) \,\backslash\, \CC{\alpha-1}{k_\alpha-1}(\mbf{k}). 
\label{eq:BtoCLeft}
\intertext{Thus by setting $\CC{0}{x}(\mbf{k}) = \CC{{N+1}}{x}(\mbf{k}) = \varnothing$, we may unify
the cases so that~\eqref{eq:BtoCLeft}
holds for all $1 \le \alpha \le N+1$. By~\eqref{eq:BtoCLeft}, the left-hand side in the lemma is}
&\hspace*{-0.65in}|\mbf{k}| \sum_{\mbf{i} \in \B(\mbf{k})} \PA{\vt}{\mbf{i}}
+ \sum_{\alpha=1\rule{0pt}{6.75pt}}^{N+1} \sum_{\mbf{i} \in \CCs{\alpha}{k_\alpha-1}(\mbf{k})} k_\alpha 
\PA{\vt}{\mbf{i}} \nonumber \\[-3pt]
& \hspace*{1.47in} - \sum_{\alpha=1\rule{0pt}{6.75pt}}^{N+1} \sum_{\mbf{i} \in \CCs{\alpha-1}{k_\alpha-1}(\mbf{k})} k_\alpha 
\PA{\vt}{\mbf{i}} .\label{eq:L}
\intertext{Similarly to~\eqref{eq:BtoCLeft} we have}
\B(\mbf{k}) - \mbf{u}^{(\beta)} &= [k_1,k_2) \times \cdots \times [k_\beta\!-\!1,k_{\beta+1}\!-\!1) \times \cdots
\times [k_N,k_{N+1})
\nonumber
\\ &= \B(\mbf{k}) \cup \CC{\beta}{k_\beta-1}(\mbf{k}) \,\backslash\, \CC{\beta}{k_{\beta+1}-1}(\mbf{k}).
\label{eq:BtoCRight}
\end{align}
By~\eqref{eq:BtoCRight}  the 
first summand in the right side in the lemma is
\[ \begin{split}
\sum_{\beta=1\rule{0pt}{6.75pt}}^N \Bigl( \sum_{\mbf{j} \in \B(\mbf{k})}
(j_\beta +1)\PA{\vt}{\mbf{j}} &+ \!\!\!\! \sum_{\mbf{j} \in \CCs{\beta}{k_\beta-1}(\mbf{k})}\!\!\!(k_\beta\!-\!1\!+\!1)\PA{\vt}{\mbf{j}} \\[-3pt]
& \quad\ \ - \!\!\!\! \sum_{\mbf{j} \in \CCs{\beta}{k_{\beta+1}-1}(\mbf{k})} \!\!\!(k_{\beta+1}\!-\!1\!+\!1) 
\PA{\vt}{\mbf{j}}\Bigr). \end{split}
\]
Since $\sum_{\beta=1}^N (j_\beta+1) = |\mbf{j}| + N$, and the 
second summand on the right-hand side is $\sum_{\mbf{j} \in \B(\mbf{k})} 
\bigl( |\mbf{k}|-N-|\mbf{j}| \bigr)\PA{\vt}{\mbf{j}}$,
the right-hand side in the lemma simplifies to
\begin{equation}\label{eq:R}
\begin{split}
|\mbf{k}| \sum_{\mbf{j} \in \B(\mbf{k})} \PA{\vt}{\mbf{j}}
&+ \sum_{\beta=1\rule{0pt}{6.75pt}}^N \sum_{\mbf{j} \in \CCs{\beta}{k_\beta-1}(\mbf{k})}  k_\beta \PA{\vt}{\mbf{j}}
\\[-3pt]
&\qquad\qquad\  - \sum_{\beta=1\rule{0pt}{6.75pt}}^N \sum_{\mbf{j} \in \CCs{\beta}{k_{\beta+1}-1}(\mbf{k})}  k_{\beta+1} \PA{\vt}{\mbf{j}}. \end{split}
\end{equation}
The lemma now follows by comparing~\eqref{eq:L} and~\eqref{eq:R}; since $\CC{0}{x}(\mbf{k}) = \CC{{N+1}}{x}(\mbf{k}) = \varnothing$ the three summands agree in the order
written.
\end{proof}

\subsubsection*{The map $\phi$ commutes with $e$}

The Lie algebra element $e \in \sl_2(\C)$  (which may be thought of
as $X \frac{\mathrm{d}}{\mathrm{d}Y}$)
acts on $\Sym^d \!E$ by $e\cdot X^{d-j}Y^j = jX^{d-j+1}Y^{j-1}$.

\begin{lemma}\label{lemma:e}
The map $\phi$ defined over the complex numbers commutes with the Lie algebra action of $e \in \sl_2(\C)$.
\end{lemma}

\begin{proof}
We compare $e \cdot \phi(x)$ and $\phi(e \cdot x)$ 
for $x$ in the canonical basis of $\Sym^{N-1}\!E \otimes \bigwedge^{N+1}
\Sym^{d+1}\!E$. 
Let $s \in \{0,1,\ldots, N-1\}$ and let $\mbf{k} \in \I{d+1}{N+1}$ be strictly increasing. For ease of notation
we set $w = s + |\mbf{k}| - N$. By~\eqref{eq:tensorAction}
and~\eqref{eq:eOnFWedge} and the definition of $\phi$ in Definition~\ref{defn:phi},
then the technical lemma to obtain the third equality, and finally~\eqref{eq:tensorAction} 
and~\eqref{eq:eOnFWedge} again
we have
\begin{align*}
&\phi\bigl( e \cdot (X^{N-1-s}Y^s \otimes F_\wedge^{(d+1)}(\mbf{k}) \bigr) \\ &\quad = \phi\bigl(
sX^{N-s}Y^{s-1} \otimes F^{(d+1)}_\wedge(\mbf{k})
+ X^{N-1-s}Y^s \otimes \sum_{\alpha=1}^{N+1} k_\alpha F^{(d+1)}_\wedge (\mbf{k} - \mbf{u}^{(\alpha)}) \bigr) 
\\
&\quad = s \sum_{\mbf{i} \in \B(\mbf{k})} \PA{s+|\mbf{k}| - 1 -N}{\mbf{i}} \\[-9pt]
& \hspace*{1in} + \sum_{\alpha=1}^{N+1} k_\alpha \sum_{\mbf{i} \in B(\mbf{k} - \mbf{u}^{(\alpha)})} 
\PA{s+|\mbf{k}| - 1 -N}{\mbf{i}}
\\[-9pt]
&\quad = 
\sum_{\beta=1\rule{0pt}{7pt}}^N \sum_{\mbf{j} \in B(\mbf{k})- \mbf{u}^{(\beta)}} (j_\beta+1)
\PA{w-1}{\mbf{j}} \\ &\hspace*{1in} + \sum_{\mbf{j} \in B(\mbf{k})} \bigl( s + |\mbf{k}| - N - |\mbf{j}| \bigr)
\PA{w-1}{\mbf{j}} \\
&\quad =
\sum_{\beta=1\rule{0pt}{6pt}}^N \sum_{\mbf{i} \in B(\mbf{k})} i_\beta
\P{w - 1}{\mbf{i} \!-\! \mbf{u}^{(\beta)}}{w\!-\!|\mbf{i}|}  
+\!\!\! \sum_{\mbf{j} \in B(\mbf{k})} \!\!\bigl( s \!+\! |\mbf{k}| \!-\! N \!-\! |\mbf{j}| \bigr)
\PA{w - 1}{\mbf{j}} \\
&\quad = e \cdot \sum_{\mbf{j} \in B(\mbf{k})} \PA{s + |\mbf{k}|-N}{\mbf{j}}  \\
&\quad = e \cdot \phi\bigl( X^{N-1-s}Y^s \otimes F_\wedge^{(d+1)}(\mbf{k}) \bigr) .
\end{align*}
The hypothesis of the technical lemma that $|\mbf{k}| - N - 1 \le w-1 \le |\mbf{k}| - 2$
follows easily from the definition of $w$.
\end{proof}

\subsubsection*{Duality}
To show that $\phi$ commutes with $f$ we use a duality argument. This appears to the authors to be more conceptual
and involve less calculation than adapting the proof already given for $e$, although this would also be possible.
Let $\mbf{e} = (d+1,\ldots, d+1) \in \I{d+1}{N+1}$ and define 
$\tau \in \End_\F\bigl( \Sym^{N-1}E \otimes \bigwedge \Sym^{N+1}\Sym^{d+1}(E) \bigr)$ by linear extension~of
\begin{align*}
 \tau \bigl( X^{N-1-j}Y^j \otimes F_\wedge^{(d+1)}(\mbf{i}) \bigr)
&= X^j Y^{N-1-j} \otimes F_\wedge^{(d+1)}(\mbf{e} - \mbf{i}) .
\intertext{Let $\mbf{d} = (d,\ldots, d) \in \I{d}{N}$ and
define $\tau' \in \End_\F\bigl( \bigwedge^N\hskip-0.5pt \Sym^d \otimes \Sym^d E \bigr)$
by linear extension of}
\tau' \bigl( F_\wedge^{(d)}(\mbf{j}) \otimes X^{d-\ell}Y^\ell  \bigr) &= 
F_\wedge^{(d)} (\mbf{d} - \mbf{j}) \otimes X^\ell Y^{d-\ell}. \end{align*}

\begin{lemma}\label{lemma:intertwining}
We have $e \tau  =  \tau f$, $\tau' e = f \tau'$ and $\tau' \phi = \pm \phi \tau$.
\end{lemma}

\begin{proof}
Observe that $\tau$ and $\tau'$ are defined by multilinear extension of the 
maps $\theta_c : \Sym^c \E \rightarrow \Sym^c \E$ defined on the canonical basis by
\[ \theta_c (X^{c-j}Y^j) = X^j Y^{c-j}. \]
Since 
$\theta_c \bigl(e  \cdot X^{c-j}Y^j\bigr) = \theta_c (j X^{c-j+1} Y^{j-1}) = j X^{j-1} Y^{c-j+1} =
f \cdot X^j Y^{c-j} = f \cdot \bigl(\theta_c (X^{c-j}Y^j) \bigr)$
we have $\theta_c\, e = f \theta_c$.
By multilinearity, this implies the first two equations in the lemma. For the third,
let $\epsilon_R$ denote the sign of the permutation of $\{1,\ldots, R\}$ reversing
the positions in an $R$-tuple. (Thus $\epsilon_R = -1$ if $R \equiv 2, 3$ mod $4$,
and otherwise $\epsilon_R = 1$.)
Let $s \in \{0,1,\ldots, N-1\}$, let $\mbf{k} \in \I{d+1}{N+1}$ be strictly
increasing, and let $\mbf{k}^\rev = (k_{N+1},\ldots, k_1)$ be the reverse of $\mbf{k}$.
Observe that $\mbf{e} - \mbf{k}^\rev$ is strictly increasing and
$|\mbf{e}-\mbf{k}^\rev| = (d+1)(N+1) - |\mbf{k}|$.
Set
\begin{equation} w = N-1-s - |\mbf{e} - \mbf{k}^\mathrm{rev}| 
= dN + N + d -s -|\mbf{k}|. \label{eq:wIntertwining} \end{equation}
By~\eqref{eq:box},
\[ \B(\mbf{e}-\mbf{k}^\rev) = [d+1-k_{N+1}, d+1-k_N) \times \cdots \times [d+1-k_2, d+1-k_1). \]
Thus $(i_1,\ldots,i_N) \in \B(\mbf{e}-\mbf{k}^\rev)$ if and only if
$(d+1-i_N, \ldots, d+1-i_1) \in (k_1,k_2] \times \cdots \times (k_N,k_{N+1}]$,
so if and only if 
$(d-i_N,\ldots, d-i_1) \in \B(\mbf{k}) = [k_1,k_2) \times \cdots \times [k_N,k_{N+1})$. 
Using this to step from line 3 to line 4 below, and the
alternative definition of $\phi$ in~\eqref{eq:phiAlt} for the immediately preceding step, we have
\begin{align*}
\phi \tau  \bigl(& X^{N-1-s}Y^s \otimes F^{(d+1)}_\wedge(\mbf{k}) \bigr) \\ 
&\quad = \phi \bigl(X^s Y^{N-1-s} \otimes \epsilon_{N+1} 
F^{(d+1)}_\wedge (\mbf{e}-\mbf{k}^\rev) \bigr) \\
&\quad = \sum_{\mbf{i} \in \B(\mbf{e} - \mbf{k}^\rev)} \epsilon_{N+1} 
F(\mbf{i}, w - |\mbf{i}|) \\ 
&\quad = \epsilon_{N+1}\sum_{\mbf{j} \in \B(\mbf{k})}  
F\bigl( (d-j_N, \ldots, d-j_1), w - (dN - |\mbf{j}|) \bigr) \\
&\quad = \epsilon_N\epsilon_{N+1} \tau' \bigl( \sum_{\mbf{j} \in \B(\mbf{k})} 
F\bigl( (j_1,\ldots, j_N), d- (w-(dN - |\mbf{j}|)) \bigr)  \\
&\quad = \epsilon_N\epsilon_{N+1} \tau' \bigl( \sum_{\mbf{j} \in \B(\mbf{k})} 
F\bigl( (j_1,\ldots, j_N), s + |\mbf{k}| - N - |\mbf{j}| \bigr) \bigr) \\
&\quad = \epsilon_N\epsilon_{N+1} \tau' \phi \bigl( X^{N-1-s} Y^s \otimes F_\wedge^{(d+1)}(\mbf{k}) \bigr) 
\end{align*}
where the penultimate equality uses~\eqref{eq:wIntertwining}.
Since $\epsilon_N \epsilon_{N+1} \in \{-1,1\}$ only depends on $N$, this completes the proof.
\end{proof}

\begin{proposition}The map $\phi$ defined over the complex numbers is an $\sl_2(\C)$-homomorphism.
\end{proposition}

\begin{proof}
By Lemma~\ref{lemma:e}, $\phi$ commutes with the Lie algebra 
action of $e$. By 
this lemma and Lemma~\ref{lemma:intertwining} we have
\[ \phi f = \phi \tau \tau \!f = \phi \tau e \tau = \pm \tau' \phi e \tau = \pm \tau' e \phi \tau = 
\pm f \tau' \phi \tau
= f \phi \tau \tau = f \phi, \]
and so $\phi$ also commutes with the Lie algebra action of $f$. Since $\sl_2(\C)$ is generated by~$e$ and~$f$
the proposition follows.
\end{proof}

\subsection{The map $\phi$ is an $\SL_2(\F)$-isomorphism}
Fix $N \in \N$ and $d \in \N_0$. The canonical basis of  
$\Sym^{N-1}\E \,\otimes\, \bigwedge^{N+1}\Sym^{d+1}\E$
is indexed by pairs $(s,\mbf{k})$ with 
$s \in \{0,1,\ldots, N-1\}$ and $\mbf{k} \in \I{d+1}{N+1}$ strictly increasing. Whenever we write a pair $(s,\mbf{k})$,
it satisfies these conditions. By~\eqref{eq:phiAlt}, the vectors
\begin{equation} v_{(s,\mbf{k})} = \sum_{\mbf{i} \in \B(\mbf{k})} 
F(\mbf{i}, w - |\mbf{i}|)
\label{eq:v} \end{equation} 
where $w = s + |\mbf{k}| - N$ 
are the images under $\phi$ of the canonical basis of 
$\Sym^{N-1}\E \otimes \bigwedge^{N+1}\Sym^{d+1}\E$.

By Lemma~\ref{lemma:imageInDelta} $\phi$ has image contained in 
$\Delta^{(2,1^{N-1})}\Sym^d E$ and
by Proposition~\ref{prop:dims}, this space has
 the same dimension as the domain of $\phi$. Therefore
to complete the proof that $\phi$ is an isomorphism of $\SL_2(\F)$-representations, it suffices to show
that the vectors $v_{(s, \mbf{k})}$ span $\Delta^{(2,1^{N-1})}\Sym^d E$.

\subsubsection*{Preliminary results on chains}
Our proof of this requires a close analysis of the neighbour
map in Definition~\ref{defn:neighbour} and the chains in Lemma~\ref{lemma:chain}.
Recall from Definition~\ref{defn:neighbour} that the content
of a pair $(\mbf{i}, j)$ is the multiset $\{i_1,\ldots, i_N\} \cup \{j\}$.

\begin{lemma}\label{lemma:boxChain}
Suppose that $(\mbf{i}, j) \in \I{d}{N} \times \{0,1,\ldots, d\}$ is a semistandard
pair such that $F(\mbf{i}, j)$ has a non-zero coefficient in
$v_{(s, \mbf{k})}$ and such that $j \not \in \{i_1, \ldots, i_N\}$. Let $(\mbf{i'},j') = \NB^m(\mbf{i},j)$ for some $m \in \N$, and suppose
that $F(\mbf{i'},j')$ has a non-zero coefficient in $v_{(s,\mbf{k})}$. Then
 $m=1$.

\end{lemma}
\begin{proof}
    By the definition of the neighbour map, 
    $\mbf{i'}$ is strictly increasing, and since $(\mbf{i},j)$ and $(\mbf{i'},j')$ have the same content, we may write
\begin{align*}
(\mbf{i}, j)    &= \bigl( (i_1,\ldots, i_{\alpha-1}, i_\alpha, i_{\alpha+1}, \ldots, i_N), j \bigr) \\
(\mbf{i'}, j') &= \bigl( (i_1,\ldots, i_{\alpha-1}, j, i_{\alpha+1}, \ldots, i_N), i_\alpha \bigr)
\end{align*}
for some $\alpha \in \{1,\ldots,N\}$. Since $j\not \in \{i_1,\ldots, i_N\}$, by Lemma~\ref{lemma:chain}, the neighbour map $\NB$ is injective on multi-indices of content $\{i_1,\ldots, i_N, j\}$.
 Therefore it suffices to prove that $(\mbf{i'},j') = \NB (\mbf{i},j)$. To see this, we need to show that $\alpha$ is maximal such that $i_\alpha \le j$. It is clear from Lemma~\ref{lemma:chain} that
  repeated applications of $\NB$ replace entries in the first element of the semistandard pair $(\mbf{i},j)$ by larger or equal entries. Thus, $j \ge i_\alpha$. If $\alpha = N$, we are done. Otherwise, since $\mbf{i'} \in \B(\mbf{k})$, we have $ j<k_{\alpha+1} \le i_{\alpha+1}$. Since $i_\alpha \le i_{\alpha+1}$, the result now follows. 
\end{proof}

By Proposition~\ref{prop:dims}, $\Delta^{(2,1^{N-1})}\Sym^d E$ has as
a basis the $F_\Delta( \mbf{i}, j )$ for semistandard pairs
$(\mbf{i}, j) \in \I{d}{N} \times \{0,1,\ldots, d\}$.

\begin{lemma}\label{lemma:atMostOneDeltaPerContent}
In the expression of $v_{(s, \mbf{k})}$ as a linear combination of the
$F_\Delta(\mbf{i}, j)$ for semistandard pairs
$(\mbf{i}, j) \in \I{d}{N} \times \{0,1,\ldots, d\}$
 there is at most one $(\mbf{i}, j)$ of any given content.
\end{lemma}

\begin{proof}
If $j \in \{i_1,\ldots, i_N\}$ then 
the unique basis element from the $F_\Delta$ basis whose semistandard pair has  content $\{i_1,\ldots, i_N\} \cup \{j\}$
is $F_\Delta(\mbf{i},j)$; since this
element is $F(\mbf{i}, j)$, in this case, the lemma obviously holds.

In the remaining case, the content multiset is a set, and by Lemma~\ref{lemma:chain}, there are $N$ semistandard pairs of content $\{i_1,\ldots, i_N\} \cup \{j\}$
forming a chain
\begin{equation}\label{eq:chain2} 
(\mbf{i}^{(1)}, j^{(1)}), \ldots, (\mbf{i}^{(\beta)}, j^{(\beta)}),  \ldots, 
(\mbf{i}^{(\gamma)}, j^{(\gamma)}), \ldots, (\mbf{i}^{(N)}, j^{(N)}) \end{equation}
with $d \ge j^{(1)} > \ldots > j^{(N)} \ge 1$. (To be explicit, Lemma~\ref{lemma:chain}
gives $j^{(1)} = j$
and $j^{(\alpha)} = i_{N-\alpha+2}$ for $2 \le \alpha \le N$.)
Choose $\beta$ minimal such that $F_\Delta(\mbf{i}^{(\beta)}, j^{(\beta)})$
has a non-zero coefficient in $v_{(s, \mbf{k})}$ and 
$\gamma$ maximal such that $F_\Delta(\mbf{i}^{(\gamma)}, j^{(\gamma)})$
has a non-zero coefficient in $v_{(s, \mbf{k})}$.

Observe that $F(\mbf{i}^{(\beta)},j^{(\beta)})$ and 
$F\bigl(\NB(\mbf{i}^{(\gamma)},j^{(\gamma)})\bigr) 
= F\bigl(\NB^{\gamma-\beta+1}(\mbf{i}^{(\beta)},j^{(\beta)})\bigr)$ 
both have a non-zero coefficient in $v_{(s,\mbf{k})}$ in the $F$ basis. By Lemma~\ref{lemma:boxChain} applied to $(\mbf{i}^{(\beta)},j^{(\beta)})$, we have  $\beta = \gamma$, as required.
\end{proof}

\subsubsection*{Triangularity}
It will be useful to denote the content multiset of a semistandard pair $(\mbf{i}, j)
\in \I{d}{N} \times \{0,1,\ldots, N\}$ by $\CN(\mbf{i}, j)$.
Given distinct multisets $\{a_1,\ldots, a_{N+1}\}$ and $\{b_1,\ldots, b_{N+1}\}$ written
so that $a_1 \le \ldots \le a_{N+1}$ and
$b_1 \le \ldots \le b_{N+1}$, we order them  
lexicographically so that
\[ \{a_1,\ldots, a_{N+1} \} < \{b_1, \ldots, b_{N+1} \} \]
if and only if $a_1 = b_1, \ldots, a_{\gamma-1} = b_{\gamma-1}$ 
and $a_\gamma < b_\gamma$, where $\gamma$ is minimal
such that $a_\gamma \not= b_\gamma$.
Using this we define a total order $\preceq$ on semistandard pairs.

\begin{definition}
Let $\prec$ be the total order on semistandard pairs in
$\I{d}{N} \times \{0,1,\ldots, d\}$
defined by  
$(\mbf{i}, j) \prec (\mbf{i}', j')$ if and only if \emph{either}
$\CN(\mbf{i}, j) < \CN(\mbf{i}', j')$
\emph{or} $\CN(\mbf{i}, j) = \CN(\mbf{i}', j')$ and $j > j'$.
\end{definition}

For example the least element under $\preceq$ is $\bigl( (0,1,\ldots, N-1), 0\bigr)$
and the greatest is $\bigl( (d-N+1,\ldots, d), d \bigr)$.
The final condition $j > j'$ is chosen so that the chain in~\eqref{eq:chain2} is
strictly increasing in the $\preceq$ total order. This can be seen in the following
example.

\begin{example}\label{ex:triangular}
Take $N = 2$ and $d = 4$. The map
$\phi$
takes each canonical basis element $v_{(s, \mbf{k})} 
\in E \otimes \bigwedge^3\Sym^4 \E$ of $Y$-degree $9$ to 
a sum of canonical basis elements of $\bigwedge^2 \Sym^4\E \otimes \Sym^4\E$
each of $Y$-degree $7$. 
The relevant $(\mbf{i}, j)$ labelling the rows of the matrix
of $\phi$ restricted to the weight space have $|\mbf{i}| + j = 7$ and 
are totally ordered
\[ \cp{(0,3)}{4} \prec\cp{(0,4)}{3} \prec\cp{(1,2)}{4} \prec \cp{(1,4)}{2}  \prec \cp{(1,3)}{3}
\prec \cp{(2,3)}{2} \]
by $\preceq$.
By Lemma~\ref{lemma:atMostOneDeltaPerContent}, at
most one $F_\Delta(\mbf{i}, j)$ of any given content appears in each row.
This can be seen in the two $2 \times 2$ identity blocks in the relevant
block of the matrix of~$\phi$ shown below.
We deliberately use a non-standard notation in which 
$\cdot$ denotes an entry known to be zero by Lemma~\ref{lemma:atMostOneDeltaPerContent},
 and empty spaces
denote zeros from the lower-triangularity implied by the following proposition.

\begin{center}\let\quad\bigquad
\setlength\arraycolsep{3.7pt}
\begin{tikzpicture}[x=1cm,y=-1cm]
		\node at (0.85,0){$\begin{matrix} \ddp{1}{(0,3,5)} & \ddp{0}{(0,4,5)} 
		& \ddp{1}{(1,2,5)} & \ddp{0}{(1,3,5)} 
		& \ddp{1}{(1,3,4)} & \ddp{0}{(2,3,4)} \end{matrix}$};
		\node at (0,2){$
\bordermatrix{ & \cr
	\cp{(0,3)}{4} & 1                             \cr 
	\cp{(0,4)}{3} & \cdot & 1                     \cr 
	\cp{(1,2)}{4} & 0     & 0 & 1                 \cr
	\cp{(1,4)}{2} & 1     & 1 & \cdot & 1         \cr
	\cp{(1,3)}{3} & 1     & 0 & 1     & 1 & 1     \cr
	\cp{(2,3)}{2} & 1     & 0 & 0     & 1 & 1 & 1 \cr}$};
\end{tikzpicture}
\end{center}
\end{example}


\begin{proposition}\label{prop:triangular}
Let $(\mbf{i}, j) \in \I{d}{N} \times \{0,1,\ldots, d\}$ be semistandard.
Let $\alpha$ 
be maximal such that $i_\alpha \le j$. 
Set $\mbf{k} = (i_1, \ldots, i_\alpha, j+1, i_{\alpha+1}+1, \ldots, i_N + 1)$.
Then
\[ v_{(\alpha-1, \mbf{k})} 
= F_\Delta(\mbf{i}, j) + v \]
where $v$ is a linear combination of basis elements $F_\Delta(\mbf{i}',j')$
for semistandard $(\mbf{i}', j')$ such that $(\mbf{i}', j') \succ (\mbf{i}, j)$.
\end{proposition}

\begin{proof}
The quantity $s + |\mbf{k}| - N - |\mbf{i}|$ in the definition of $\phi$
in~\eqref{eq:phiAlt} is 
\[ (\alpha-1) + |\mbf{k}| - N - |\mbf{i}| = (\alpha-1) + (|\mbf{i}| + N - \alpha + j +1) - N - |\mbf{i}|
= j. \]
Since
\begin{equation}\label{eq:iDefn} i_1 = k_1, \ldots, i_\alpha = k_\alpha, i_{\alpha+1} = k_{\alpha+2}-1, \ldots, i_N = k_{N+1}-1 \end{equation}
we have $\mbf{i} \in \B(\mbf{k})$ and so by~\eqref{eq:phiAlt},
$F( \mbf{i}, j )$
has coefficient $1$
in the expression of  $ v_{(\alpha-1, \mbf{k} )}$ in the $F$ basis.
Moreover, we have
\[ \NB(\mbf{i}, j) = \bigl( (i_1, \ldots, i_{\alpha-1}, j, i_{\alpha+1}, \ldots, i_N), i_\alpha \bigr) \] 
and since $k_\alpha = i_\alpha \le j = k_{\alpha+1}-1$,
we also have $(i_1, \ldots, i_{\alpha-1}, j , i_{\alpha+1}, \ldots, i_N) \in \B(\mbf{k})$.
Therefore $F\bigl( \NB(\mbf{i}, j) \bigr)$ has coefficient $1$
in the expression of  $ v_{(\alpha-1, \mbf{k} )}$ in the $F$ basis.
It follows by Lemma~\ref{lemma:atMostOneDeltaPerContent}
that $F_\Delta(\mbf{i}, j)$ has coefficient~$1$ when $v_{(\alpha-1, \mbf{k})}$ is written
in the $F_\Delta$ basis, and 
the remaining summands in the $F_\Delta$ basis have different content to $(\mbf{i},j)$.

Suppose for contradiction that such a pair $(\mbf{i}',j')$ satisfies $(\mbf{i}',j') \prec (\mbf{i},j)$, or equivalently, since $\CN(\mbf{i}',j') \not= \CN(\mbf{i},j)$, 
that $\CN(\mbf{i}',j') < \CN(\mbf{i},j)$. Let
\begin{align*}
    \mbf{c}&=(c_1,\ldots,c_{N+1}) = (i_1, \ldots, i_{\alpha}, j , i_{\alpha+1}, \ldots, i_N) \\
    \mbf{c'}&=(c'_1,\ldots,c'_{N+1}) =(i_1',\ldots,i'_{\beta}, j', i'_{\beta+1}, \ldots, i_N')
\end{align*}
where $\beta$ is maximal such that $i'_\beta \le j'$. 
 Thus both $\mbf{c}$ and $\mbf{c'}$ are (weakly) increasing.


Suppose first of all 
that $\alpha < \beta$. 
From~\eqref{eq:iDefn} we have $i_\gamma = k_\gamma$ if $1 \le \gamma \le \alpha$
and by definition of $\mbf{k}$, we have $k_{\alpha+1} = j + 1$.
Thus $\mbf{i'} \in \B(\mbf{k})$ implies that 
$c'_\gamma = i'_\gamma \ge k_\gamma = i_\gamma = c_\gamma$ 
for $1 \le \gamma \le \alpha$ 
and $c'_{\alpha+1} = i'_{\alpha+1} \ge  k_{\alpha+1} = j+1$. 
Since we are assuming that $c(\mbf{i'},j') < c(\mbf{i},j)$, we must have $c'_\gamma = c_\gamma$ for $1 \le \gamma \le \alpha$. However,
$c'_{\alpha+1} \ge j+1 > j = c_{\alpha+1}$. 
Thus $\CN(\mbf{i'},j') > \CN(\mbf{i},j)$, a contradiction.

In the remaining case we have $\alpha \ge \beta$. 
Similarly to the first case, we have
$c'_\gamma = i'_\gamma \ge k_\gamma = i_\gamma =  
c_\gamma$ for $1 \le \gamma \le \beta$, and thus
our assumption that $\CN(\mbf{i}', j') < \CN(\mbf{i}, j)$  implies 
that $c'_\gamma = c_\gamma$ for $1\le \gamma \le \beta$ and $c'_{\beta+1} \le c_{\beta+1}$. 
Since $\CN(\mbf{i}, j) \not= \CN(\mbf{i'}, j')$ 
and $|\mbf{i}|+j = |\mbf{i'}|+j'$, there exists 
$\delta$ such that $\beta+2 \le \delta \le N+1$ 
and $c'_\delta > c_\delta$.
Now, since $\delta \ge \beta+2$, we have $c'_\delta 
= i'_{\delta -1}$. In turn, $\mbf{i'}\in \B(\mbf{k})$ implies $i'_{\delta -1} < k_\delta$. By definition of $\mbf{k}$, we also have  
$k_\gamma  \in \{c_\gamma, c_\gamma+1\}$ for $1 \le \gamma \le N+1$. In particular,
$k_\delta \le c_\delta + 1$. Putting all these together, we have \[c'_\delta = i'_{\delta-1} < k_\delta \le c_\delta+1\]
and hence $c'_\delta \le c_\delta$,
a final contradiction.

Therefore, any $F_\Delta(\mbf{i}', j')$ appearing with non-zero coefficient
in the expression of $v_{(\alpha-1, \mbf{k})}$ in the $F_\Delta$ basis
satisfies $(\mbf{i}', j') \succ (\mbf{i}, j)$, as required.
\end{proof}

We can now easily complete the proof of our main theorem.

\begin{corollary}
The vectors $v_{(s, \mbf{k})}$ for $s \in \{0,1,\ldots, N-1\}$
and $\mbf{k} \in \I{d+1}{N+1}$
span $\Delta^{(2,1^{N-1})}\Sym^d\E$. 
\end{corollary}

\begin{proof}
We order the codomain basis elements from least to greatest under $\preceq$, and we order the domain basis elements so that if $F_\Delta(\mathbf{i}, j)$ is in the $m^\mathrm{th}$ position in the codomain order, then its corresponding $X^{N-1-(\alpha-1)}Y^{\alpha -1} \otimes F_{\wedge}^{(d+1)}(\mathbf{k})$ (where $\alpha$ and $\mbf{k}$ are as defined in Proposition~\ref{prop:triangular}) is in the $m^\mathrm{th}$ position in the domain order. (This is illustrated in Example~\ref{ex:triangular}.)

With this ordering, by  Proposition~\ref{prop:triangular}, the matrix expressing each $v_{(s, \mathbf{k})}$ in the $F_\Delta(\mathbf{i}, j)$ basis is lower uni-triangular. Its inverse is also uni-triangular, and it expresses each $F_\Delta(\mathbf{i}, j)$ as a linear combination of the images $v_{(s, \mathbf{k})}$. By Proposition~\ref{prop:dims}, the elements $F_\Delta(\mathbf{i}, j)$ form a basis for $\Delta^{(2,1^{N-1})}\Sym^d\E$, so we are done.
\end{proof}

This completes the proof that $\phi$ is surjective, which concludes the proof of Theorem~\ref{thm:main}.

\section{Final remarks}

In this section we first show that the $\SL_2(\F)$ 
isomorphism in Theorem~\ref{thm:main} becomes a $\GL_2(\F)$-isomorphism provided a suitable
power of the determinant is introduced. 
(This is typical of the general theory: see \cite[\S 3.3]{PagetWildonSL2}.)
We then obtain identity~\eqref{eq:qCharactersSame} by taking characters.
We finish with
a conjectured generalization of Theorem~\ref{thm:main}.

We denote the $1$-dimensional determinant representation of $\GL_2(\F)$ by~$\det$
and regard the domain and codomain of $\phi$ as representations of $\GL_2(\F)$ in the obvious way.

\begin{corollary}\label{cor:phiGL}
Let $N \in \N$ and let $d \in \N_0$.
The map $\phi$ defined in Definition~\ref{defn:phi} is an
isomorphism of $\GL_2(\F)$-representations
\[ \Sym^{N-1}\E \otimes \bigwedge^{N+1} \Sym^{d+1}\E \cong \det^N \otimes\, \Delta^{(2,1^{N-1})} \Sym^d\E.\]
\end{corollary}

\begin{proof}
Let $K$ be the algebraic closure of  $\F$.
Let $\widetilde{E} =\E \otimes_\F K$. It is sufficient to prove that the map 
\[ \widetilde{\phi} : \Sym^{N-1}\!\widetilde{E}\otimes \Sym^{N+1} \Sym^{d+1}\!\widetilde{E} \rightarrow  
 \det^N \otimes \, \Delta^{(2,1^{N-1})}\Sym^{d}\!\widetilde{E}\]
is a $\GL_2(K)$-isomorphism, since $\phi$ is defined with coefficients in the prime subfield 
 of $\F$,
and so via the inclusion $E \mapsto E \otimes 1 \subseteq E \otimes_\F K$, the map $\widetilde{\phi}$ restricts to $\phi$.
By Theorem~\ref{thm:main} for the field $K$, the map $\widetilde{\phi}$ is an $\SL_2(K)$-homomorphism. 
Now, because $K$ is algebraically closed, and so every element of $K$ has a square root in $K$, we have
\[ \GL_2(K) = \left\langle \SL_2(K), \left(\begin{matrix} \alpha & 0\\ 0 &\alpha \end{matrix} \right) : \alpha \in K
\backslash \{0\} \right\rangle .\]
It therefore suffices to prove that $\widetilde{\phi}$ commutes with the action of the diagonal matrices
$\alpha I$ for $\alpha \in K$. Using the canonical bases of the domain and codomain of $\widetilde{\phi}$, 
one  sees that
on the domain $\alpha I$ acts as $\alpha^{N-1 + (N+1)(d+1)} = \alpha^{(N+1)d + 2N}$
and on the codomain $\alpha I$ acts as $\alpha^{(N+1)d} \det(\alpha I)^N \!=\! \alpha^{(N+1)d} \alpha^{2N}$.
Since the exponents agree, this completes the proof.
\end{proof}

We now prove  identity~\eqref{eq:qCharactersSame}. Recall
that $s_\lambda$ is the Schur function canonically labelled by the partition $\lambda$.
It is immediate from the combinatorial definition of Schur functions (see for instance
\cite[Definition 7.10.1]{StanleyII}) that
$s_\lambda(1,q,\ldots, q^d)$ is 
the generating function enumerating semistandard tableaux of shape $\lambda$
with entries from $\{0,1,\dots, d\}$ by their sum of entries.
This gives a combinatorial interpretation of the right-hand side in~\eqref{eq:qCharactersSame}
and in Corollary~\ref{cor:qChar} below.
One of the most natural interpretations of the $q$-binomial coefficient $\qbinom{a}{b}_q$ 
is that $q^{\frac{b(b-1)}{2}} \qbinom{a}{b}_q$ is the generating function enumerating
$b$-subsets of $\{0,\ldots, a-1\}$ by their sum of entries. Thus, by identifying semistandard tableaux
of shape $(1^{N+1})$ with the subset of their entries, we deduce that
\begin{equation}
\label{eq:qBinomial} 
q^{\frac{(N+1)N}{2}} \qbinom{d+2}{N+1}_q = s_{(1^{N+1})}(1,q,\ldots, q^{d+1}). \end{equation}
For further background on $q$-binomial coefficients,
including the theorem that $\qbinom{a}{b}_q$ is the generating function enumerating partitions in the
$b \times (a-b)$ box by their size,
we refer the reader to \cite[\S 1.7]{StanleyI}.

\begin{corollary}\label{cor:qChar}
For any $N \in \N$ and $d \in \N_0$ we have
\[ 
q^{\frac{N(N-1)}{2}} [N]_q \qbinom{d+2}{N+1}_q = s_{(2,1^{N-1})}(1,q,\ldots, q^d).
\]
\end{corollary}

\begin{proof}
It suffices to prove the identity when $q$ is a non-zero complex number.
It is clear from the canonical basis $X^{N-1}, X^{N-1}Y, \ldots, Y^{N-1}$ of $\Sym^{N-1}\E$ 
that $[N]_q = 1+ q + \cdots + q^{N-1}$ is the
character of $\Sym^{N-1}\E$ evaluated at the diagonal matrix $D$ in $\GL_2(\C)$ with
entries $1$ and~$q$.
By \cite[(10)]{PagetWildonSL2}, 
$s_{(1^{N+1})}(1,q,\ldots, q^{d+1})$ and  $s_{(2,1^{N-1})}(1,q,\ldots, q^d)$ 
are  the
characters of the $\GL_2(\C)$-representations $\bigwedge^{N+1}\Sym^{d+1}\E$
and  $\Delta^{(2,1^{N-1})}\Sym^d \E$,
also evaluated at $D$. Therefore, by~\eqref{eq:qBinomial},
the character of $\Sym^{N-1}\E \otimes \bigwedge^{N+1}\Sym^{d+1}\E$
evaluated at $D$ is 
\smash{$[N]_q\, q^{(N+1)N/2}  \qbinom{d+2}{N+1}_q$}. But in Corollary~\ref{cor:phiGL} 
we showed that this representation is isomorphic 
to $\det^N \otimes\, \Delta^{(2,1^{N-1})}E $. Hence,
equating the character values we obtain
\[ 
[N]_q\, q^{\frac{(N+1)N}{2}} \qbinom{d+2}{N+1}_q = q^N s_{(2,1^{N-1})}(1,q,\ldots, q^d) .
\]
The result follows by cancelling $q^N$ from each side.
\end{proof}

Our main result, Theorem~\ref{thm:main},
is the special case when $M=2$ of the following conjecture.

\begin{conjecture}\label{conj:phiGeneral}
Let $M$, $N \in \N$. There is an isomorphism of $\SL_2(\F)$-representations
\[ \bigwedge^{M-1} \Sym^{M+N-3}\E \,\otimes \!\bigwedge^{M+N-1}\Sym^{M+d-1}\E \cong \Delta^{(M,1^{N-1})}\Sym^d\E. \]
\end{conjecture}

If $M=1$, then the first factor is $\F$ and since $\Delta^{(1^N)}V = \bigwedge^N V$, 
both sides in the claimed isomorphism are $\bigwedge^N \Sym^d \E$.
If $N=1$ then since $\Sym^{M-2}\E$ is $(M-1)$-dimensional,
and so $\bigwedge^{M-1} \Sym^{M-2}\E$ is the determinant representation of $\SL_2(\F)$, which is trivial,
and $\Delta^{(M)}V = \Sym_M V$,
the claimed isomorphism is $\bigwedge^M \Sym^{M+d-1} \E \cong \Sym_M \Sym^d\E$. An explicit
isomorphism from the right-hand side to the left-hand side is given by Theorem 1.4 in \cite{McDowellWildon}.
More broadly, it would be interesting to have field-independent results on
the endomorphism rings of the two sides in Conjecture~\ref{conj:phiGeneral}.

\section*{Acknowledgements}
The first author would like to thank his advisor Mikhail Khovanov for valuable discussions.
The second author gratefully acknowledges financial support from the Heilbronn Institute for Mathematical
Research, Bristol, UK. The authors are very grateful to the anonymous referee 
for an exceptionally careful reading of the first version of this paper and many helpful comments and
corrections.
The authors also thank {\'A}lvaro Guti{\'e}rrez C{\'a}ceres and Micha\l\ Szwej for helpful comments and corrections.

\end{document}